\documentclass[11pt,a4paper]{article}

\usepackage{amsmath,amsthm}

\PassOptionsToPackage{backend=biber}{biblatex}
\usepackage{cgeometry}
\usepackage{newunicodechar}
\newunicodechar{Ł}{\L}
\newunicodechar{ł}{\l}

\usepackage{tikz}
\usetikzlibrary{decorations.pathreplacing}
\allowdisplaybreaks
\title{A counterexample to the bounded mass property}
\author{Mingchen Xia \and Kewei Zhang}
\date{}

\begin{document}
	
	\maketitle
	
	\begin{abstract}
		A compact complex manifold has the \emph{bounded mass property} if,
		for one (equivalently, every) Hermitian form $\omega$, the masses
		$\int_X(\omega+\ddc\varphi)^n$ are uniformly bounded over all smooth
		$\varphi$ with $\omega+\ddc\varphi>0$.  We prove
		that this property fails on the Hopf threefold
		$(\mathbb C^3\setminus\{0\})/\langle z\mapsto\mathrm{e}^{-1}z\rangle$,
		answering a question of Boucksom--Guedj--Lu.
	\end{abstract}
	
	\section{Introduction}
	
	On a compact K\"ahler manifold $X$ of dimension $n$, Stokes' theorem
	shows that the total Monge--Amp\`ere mass
	$\int_X(\omega+\ddc\varphi)^n$ of a K\"ahler potential $\varphi$ does
	not depend on $\varphi$: it always equals $\int_X\omega^n$.  This
	invariance is a cornerstone of pluripotential theory on K\"ahler
	manifolds, which runs from the foundational work of Bedford--Taylor
	\cite{BT82} through Yau's solution \cite{Y78} of the Calabi
	conjecture \cite{Cal57} to Ko\l odziej's uniform estimates
	\cite{Kol98}; see \cite{GZ17} for a systematic account.
	
	On a general compact complex manifold, the natural background forms
	are merely Hermitian, and complex Monge--Amp\`ere equations remain
	both meaningful and useful: after the pioneering work of Cherrier
	\cite{Che87}, Tosatti--Weinkove \cite{TW10} solved the Calabi--Yau
	equation for an arbitrary Hermitian metric, and the Hermitian theory
	has since developed rapidly, see for instance,
	\cite{DK12,KN15,STW17} and the pluripotential approach of
	\cite{GL21,GL23}.  Since the Hermitian form $\omega$ is no longer
	closed, the mass $\int_X(\omega+\ddc\varphi)^n$ does depend
	on the potential $\varphi$, and its uniform control is a recurrent
	theme of this theory \cite{Chi14,DK12,KN15,GL23}.  Understanding how
	large the mass can become is thus a basic problem of pluripotential
	theory on Hermitian manifolds \cite{GL22}.  Following
	Boucksom--Guedj--Lu \cite[Definition 1.16]{BGL}, consider the upper
	Monge--Amp\`ere mass
	\begin{equation}
		\overline{\vol}(\omega)
		\coloneqq \sup\left\{
		\int_X(\omega+\ddc\varphi)^n:
		\varphi\in C^\infty(X,\mathbb R),\ \omega+\ddc\varphi>0
		\right\}.
		\label{eq:upper-volume}
	\end{equation}
	The manifold $X$ is said to have the \emph{bounded mass property} if
	\eqref{eq:upper-volume} is finite; as any two Hermitian forms on $X$
	are uniformly comparable, the property does not depend on the choice of
	$\omega$ \cite[\S1.5]{BGL}.
	
	The bounded mass property is known to hold in two important cases: in
	dimension $n\le2$, by testing against a Gauduchon metric
	\cite{Gau77}, and on manifolds of the Fujiki class
	\cite{Fuj78,Var89}; see \cite{GL22} and \cite[\S1.5]{BGL}.  When it
	holds, the finite quantity \eqref{eq:upper-volume}, the
	Monge--Amp\`ere volume of Guedj--Lu \cite{GL22}, behaves in several
	respects like the volume of a K\"ahler class; its finiteness is the starting
	point of the theory of volumes of Bott--Chern classes developed in
	\cite{BGL}, and it enters as a standing hypothesis in recent
	developments of Hermitian pluripotential theory \cite{ALS25,LiX26}.
	Whether the property holds on \emph{every} compact complex manifold
	was left open in \cite{BGL}.  The first test case is that of Hopf
	manifolds \cite{Hopf48}, the classical examples of compact
	non-K\"ahler complex manifolds, which occupy an active place in
	non-K\"ahler geometry, from locally conformally K\"ahler metrics
	\cite{OV24} to the Chern--Ricci flow \cite{TW15}.
	The bounded mass question is raised explicitly for Hopf manifolds of
	dimension at least three in \cite[Example 1.19]{BGL}; the following
	theorem answers it in the negative, disproving the universal bounded
	mass property already in complex dimension three.
	
	\begin{theorem}\label{thm:main}
		On the Hopf threefold
		\[
		X=(\mathbb C^3\setminus\{0\})/\langle z\mapsto \mathrm{e}^{-1}z\rangle
		\]
		there are a smooth Hermitian form $\omega$ and functions
		$\varphi_j\in C^\infty(X,\mathbb R)$, $j\in\mathbb N$, such that
		\[
		\omega+\ddc\varphi_j>0,
		\qquad
		\int_X(\omega+\ddc\varphi_j)^3\longrightarrow+\infty
		\quad\text{as }j\to\infty.
		\]
		Consequently $\overline{\vol}(\omega)=+\infty$, so $X$ does not have the
		bounded mass property.
	\end{theorem}
	
	\subsection*{Strategy of the proof}
	
	The starting point is the structure of the Hopf threefold as an
	elliptic fibration $\pi\colon X\to\mathbb P^2$ with fiber
	$E=\mathbb C/(\mathbb Z+2\pi\mathrm{i}\mathbb Z)$.  The threefold
	carries a natural Hermitian form $\omega=\alpha+\beta$, where
	$\alpha=\pi^*\omega_{\mathrm{FS}}$ is the pullback of a Fubini--Study
	form and $\beta$ is a semipositive form of rank one which is positive
	along the fibers.  The failure of $\omega$ to be closed is concentrated
	in the identity $\ddc\beta=-\alpha^2$, whose sign is decisive: it
	produces, for every $V\in C^\infty(X,\mathbb R)$, the exact mass
	identity
	\[
	\int_X(\omega+\ddc V)^3
	=\int_X\omega^3+12\pi\int_{\mathbb P^2}
	\left(\int_{\pi^{-1}(y)}|V_w|^2\,\mathrm{d}A_E\right)
	\omega_{\mathrm{FS}}^2(y),
	\]
	where $V_w$ is the fiberwise derivative and $\mathrm{d}A_E$ the
	normalized Haar measure of the fiber (\cref{lem:mass}).  The
	Monge--Amp\`ere mass thus exceeds $\int_X\omega^3$ by a fiberwise
	Dirichlet energy, and the theorem reduces to constructing potentials
	$V_\eta$, with $\omega+\ddc V_\eta>0$, whose fiber energy over a fixed
	ball in the base tends to infinity.  We work with potentials
	$V=V(x,w)$ depending only on the fiber variable $w$ and the logarithmic
	radial coordinate $x$ of an affine chart of the base; for these,
	positivity of $\omega+\ddc V$ reduces to three scalar inequalities
	(\cref{lem:positivity}), of which the essential one is
	\[
	Q[V]\coloneqq V_{xx}+\mu-\frac{|\mu+V_{x\bar w}|^2}{1+V_{w\bar w}}>0,
	\qquad
	\mu\coloneqq\frac{\mathrm{e}^x}{1+\mathrm{e}^x}.
	\]
	
	The Green function of each fiber has infinite
	energy, and a small negative multiple of a heat-kernel regularization
	yields potentials $F_\eta(t,\cdot)$ with energy
	$2\eta^2\log(1/t)+O(\eta^2)$ as $t\downarrow0$
	(\cref{lem:heat-fiber}).  Taking
	$\log(1/t)=\eta^{-2}\log(1/\eta)$ therefore makes the energy diverge.
	
	The global difficulty is that $V(x,w)$ must be constant for $x\gg0$ in
	order to extend smoothly to $X$: the fiber oscillation has to be
	suppressed as $x$ increases, and every change of the fiber-dependent part
	enters the positivity condition through the cross term
	$|\mu+V_{x\bar w}|^2$ in $Q[V]$.  We suppress the oscillation along
	the heat flow itself, setting $V(x,w)=F_\eta(t_\eta(x),w)$, where
	$t_\eta$ increases slowly from
	$\varepsilon_\eta$ to the order of $\eta^{-1}$; once the time exceeds
	$(\mathrm{e}\,\eta)^{-1}$ the fiber dependence is exponentially small
	and is then removed by a cutoff function.  Two properties of the heat kernel
	control the effect of the transport on $Q$.  First, the heat
	equation combined with the logarithmic convexity bound
	$(\log k_t)_{w\bar w}\ge-1/t$ controls the terms of $Q$ which are
	quadratic in the transport speed.  Second, \cref{lem:heat-transition}
	shows that the remaining
	linear term is either dominated by a nonnegative complete square or
	itself exponentially small.  The resulting negative part of $Q$ has
	$x$-integral $o(1)$ as $\eta\to0$, and adding to $V(x,w)$ a correction term
	depending on $x$ alone removes it without changing the fiber energy
	over the fixed ball.
	
	\paragraph*{Notation.}

	We write $\ddc=\mathrm{i}\partial\bar\partial$.  Unless it carries a
	subscript or is explicitly defined, $C$ denotes a positive absolute
	constant whose value may change from one occurrence to the next.

	\paragraph*{Acknowledgments.}
	
	The initial counterexample was constructed with the
	\href{https://github.com/frenzymath/Rethlas}{Rethlas agent} (improved by Felix Ye), using the
	\texttt{gpt-5.6-sol} model.  The authors then simplified the
	construction and improved the presentation, and are fully responsible
	for all assertions in this paper.
	
	The first-named author is supported by the National Key R\&D Program
	of China under Grant No.~2025YFA1018200.  The second-named author is
	supported by the Scientific Research Innovation Capability Support
	Project for Young Faculty SRICSPYF-ZY2025169 and NSFC grant 12571060.
	
	\section{The Hopf threefold and the mass identity}

	\subsection{The elliptic fibration and the background form}

	Let $\widetilde X=\mathbb C^3\setminus\{0\}$ and let
	$\gamma\colon\widetilde X\to\widetilde X$ be the contraction
	$\gamma(z)=\mathrm{e}^{-1}z$.  The quotient
	\[
		X=\widetilde X/\langle\gamma\rangle
	\]
	is a compact Hopf threefold.  Projectivization is invariant under $\gamma$
	and therefore induces an elliptic fibration
	\[
	\pi\colon X\longrightarrow\mathbb P^2,
	\qquad [z]\longmapsto[z_0:z_1:z_2].
	\]
	All the fibers are isomorphic to the elliptic curve
	\[
	E\coloneqq\mathbb C/(\mathbb Z+2\pi\mathrm{i}\mathbb Z).
	\]
	
	We first construct the background Hermitian form.  On $\widetilde X$, set
	\[
	u=\log|z|^2,
	\qquad
	\alpha=\mathrm{i}\partial\bar\partial u,
	\qquad
	\beta=\mathrm{i}\partial u\wedge\bar\partial u,
	\qquad
	\omega=\alpha+\beta.
	\]
	Although $u\circ\gamma=u-2$, the forms $\alpha$ and $\beta$ are
	$\gamma$-invariant and hence descend to $X$.  The form $\alpha$ is the
	pullback $\pi^*\omega_{\mathrm{FS}}$ for a suitable normalization of
	the Fubini--Study form $\omega_{\mathrm{FS}}$ on $\mathbb P^2$.  At a point
	$q\in X$, the vertical tangent line is $\ker(\mathrm{d}\pi_q)$.
	The form $\alpha$ has two positive eigenvalues and
	$\ker(\alpha_q)=\ker(\mathrm{d}\pi_q)$.  The form $\beta$ has one positive
	eigenvalue, and its restriction to $\ker(\mathrm{d}\pi_q)$ is positive.
	Consequently $\ker(\alpha_q)\cap\ker(\beta_q)=\{0\}$, so
	$\omega=\alpha+\beta$ is positive definite.  Moreover
	\begin{equation}
		\alpha^3=0,
		\qquad \beta^2=0,
		\qquad \mathrm{d}\alpha=0,
		\qquad \ddc\beta=-\alpha^2,
		\qquad \ddc(\omega^2)=0.
		\label{eq:hopf-identities}
	\end{equation}
	The first three identities hold because $\alpha$ is pulled back from
	the surface $\mathbb P^2$, because $\partial u\wedge\partial u=0$, and
	because $\alpha=\ddc u$.  For the fourth,
	$\bar\partial\beta=-\mathrm{i}\,\partial\bar\partial u\wedge
	\bar\partial u$, whence
	$\ddc\beta=\mathrm{i}\partial\bar\partial\beta
	=(\partial\bar\partial u)^2=-\alpha^2$.  For the last,
	$\omega^2=\alpha^2+2\alpha\wedge\beta$ by $\beta^2=0$, while
	$\ddc(\alpha^2)=0$ because $\alpha^2$ is closed of pure type and hence
	separately $\partial$- and $\bar\partial$-closed, and
	$\ddc(\alpha\wedge\beta)=\alpha\wedge\ddc\beta=-\alpha^3=0$.
	
	\subsection{The affine trivialization}

	Let
	\[
	U_0=\{[z_0:z_1:z_2]\in\mathbb P^2:z_0\ne0\}\simeq\mathbb C^2
	\]
	and write
	\[
	\xi_j=\frac{z_j}{z_0}\quad(j=1,2),
	\qquad
	\xi=(\xi_1,\xi_2)\in\mathbb C^2.
	\]
	Thus $\xi$ is the affine coordinate vector on $U_0$, and $|\xi|$ denotes
	its standard Euclidean norm.  For a point of $\pi^{-1}(U_0)$ represented
	by $z=(z_0,z_1,z_2)$, choose any number $w\in\mathbb C$ satisfying
	$\mathrm{e}^{w}=z_0$.  Another such number is $w+2\pi\mathrm{i}k$
	with $k\in\mathbb Z$, whereas replacing $z$ by $\gamma^m(z)$,
	$m\in\mathbb Z$, replaces $w$ by $w-m$.
	Thus the class
	$
	[w]\in\mathbb C/(\mathbb Z+2\pi\mathrm{i}\mathbb Z)=E
	$
	is unchanged when either the logarithm of $z_0$ or the representative
	of the Hopf class is changed.  We therefore have the explicit
	trivialization
	\begin{equation}
		\pi^{-1}(U_0)\longrightarrow U_0\times E,
		\qquad
		[z]\longmapsto\bigl(\xi,[w]\bigr),
		\label{eq:affine-trivialization}
	\end{equation}
	whose inverse sends $(\xi,[w])$ to the Hopf class of
	$\mathrm{e}^{\widetilde w}(1,\xi_1,\xi_2)$, where $\widetilde w$ is any
	representative of $[w]$.  Changing the representative does not change the
	Hopf class.  We henceforth write $w$ for the class in $E$.  Expressions
	involving $\mathrm{d}w$ or derivatives with respect to $w$ are computed
	using any local complex lift of that class and do not depend on the lift.
	We denote by
	\begin{equation}\label{eq:def-dA}
		\mathrm{d}A_E\coloneqq \frac{\mathrm{i}}{4\pi}\mathrm{d}w\wedge \mathrm{d}\bar w
	\end{equation}
	the Haar area form on $E$, chosen such that
	$\int_E\mathrm{d}A_E=1$.

	We record an elementary extension principle on
	the affine chart.

	\begin{lemma}\label{lem:extension}
		Let $K$ be a compact subset of $U_0$, and let $f$ be a smooth
		function on $\pi^{-1}(U_0)$ which is equal to a constant on
		$\pi^{-1}(U_0\setminus K)$.  Then $f$ extends to a smooth function
		on $X$.
	\end{lemma}

	\begin{proof}
		Define $\widehat f=f$ on $\pi^{-1}(U_0)$ and $\widehat f=c$ on
		$X\setminus\pi^{-1}(K)$.  These two open sets cover $X$, and the two
		definitions agree on their overlap $\pi^{-1}(U_0\setminus K)$.  Thus
		$\widehat f$ is a smooth function on $X$.
	\end{proof}

	\subsection{The positivity criterion}

	On the punctured affine chart $\xi\ne0$, set
	\[
		x=\log|\xi|^2,
		\qquad
		\mu=\mu(x)=\frac{\mathrm{e}^x}{1+\mathrm{e}^x}.
	\]
	Thus $x$ is the logarithmic radial coordinate and one has $\mu_x=\mu(1-\mu)$.
	For a smooth real function $V=V(x,w)$, subscripts denote derivatives.
	\begin{lemma}\label{lem:positivity}
		Let $V\colon\mathbb R\times E\to\mathbb R$ be smooth, and suppose
		that $V$ is independent of $x$ for all sufficiently negative $x$,
		and equal to a single constant, independent of both $x$ and $w$,
		for all sufficiently positive $x$.  Then $V$ extends to a smooth
		function on $X$, and on
		$\xi\ne0$ the Hermitian form
		$\omega+\ddc V$ is positive if and only if
		\begin{equation}\label{eq:three-minors}
			1+V_{w\bar w}>0,
			\qquad
			\mu+V_x>0,
			\qquad
			Q[V]\coloneqq V_{xx}+\mu-
			\frac{|\mu+V_{x\bar w}|^2}{1+V_{w\bar w}}>0.
		\end{equation}
		At $\xi=0$, positivity is equivalent to $1+V_{w\bar w}>0$.
		Where $V$ is constant, positivity is automatic.
	\end{lemma}
	
	\begin{proof}
		The hypothesis for $x\ll0$ gives a smooth extension across $\xi=0$,
		while the hypothesis for $x\gg 0$ and \cref{lem:extension} extend the
		function smoothly to all of $X$.
		We compute the positivity condition at a point with $\xi\ne0$.  In the
		trivialization \eqref{eq:affine-trivialization}, using a local lift of $w$,
		one has
		\[
		u=w+\bar w+\log(1+|\xi|^2)
		=w+\bar w+\log(1+\mathrm{e}^x),
		\qquad
		\partial u=\mathrm{d}w+\mu\,\partial x.
		\]
		Since the local potential of $\omega_{\mathrm{FS}}$ is
		$\log(1+|\xi|^2)$, differentiation gives
		\begin{equation}
			\alpha
			=\mathrm{i}\mu\,\partial\bar\partial x
			+\mathrm{i}\mu(1-\mu)\,\partial x\wedge\bar\partial x.
			\label{eq:alpha-x}
		\end{equation}
		We now specify the three tangent vectors used to test positivity.  Set
		\[
		e_r\coloneqq \frac{1}{\sqrt{\mu(1-\mu)}}(\xi_1\partial/\partial\xi_1+\xi_2\partial/\partial\xi_2).
		\]
		Then $e_r$ is an $\alpha$-unit vector with
		$
		\partial x(e_r)=\frac{1}{\sqrt{\mu(1-\mu)}}.
		$
		Choose an $\alpha$-unit vector $e_\perp$ in the $\alpha$-orthogonal
		complement of $e_r$ in the $\xi$-coordinate plane.  Since the
		Fubini--Study pairing of a vector $\sum_j a_j\partial/\partial\xi_j$
		with the radial vector equals $(1+|\xi|^2)^{-2}\sum_j a_j\bar\xi_j$,
		orthogonality forces $\sum_j a_j\bar\xi_j=0$, that is,
		$\partial x(e_\perp)=0$.  Finally,
		$e_w\coloneqq\partial/\partial w$ spans the vertical tangent line and
		has unit length for the Hermitian metric associated with $\omega$.
		
		In the ordered frame $(e_\perp,e_r,e_w)$, the Hermitian matrix of
		$\omega+\ddc V$ is
		\begin{equation}\label{eq:hermitian-matrix}
			\begin{pmatrix}
				\dfrac{\mu+V_x}{\mu} & 0 & 0\\[6pt]
				0 & \dfrac{\mu+V_{xx}}{\mu(1-\mu)} &
				\dfrac{\mu+V_{x\bar w}}{\sqrt{\mu(1-\mu)}}\\[9pt]
				0 & \dfrac{\mu+V_{wx}}{\sqrt{\mu(1-\mu)}} & 1+V_{w\bar w}
			\end{pmatrix}.
		\end{equation}
		Indeed, substituting \eqref{eq:alpha-x} in the calculation of
		$\omega+\mathrm{i}\partial\bar\partial V$ yields
\begin{align}
 \omega+\ddc V
 ={}&\frac{\mu+V_x}{\mu}\,\alpha
 +\mathrm{i}\bigl(\mu^2+V_{xx}-(1-\mu)V_x\bigr)
       \partial x\wedge\bar\partial x\notag\\
 &+\mathrm{i}(\mu+V_{x\bar w})
       \partial x\wedge\mathrm{d}\bar w
 +\mathrm{i}(\mu+V_{wx})
       \mathrm{d}w\wedge\bar\partial x\notag
 +\mathrm{i}(1+V_{w\bar w})
       \mathrm{d}w\wedge\mathrm{d}\bar w.
 \label{eq:omegaV-explicit}
\end{align}
        Evaluating this on
		$(e_\perp,e_r,e_w)$ gives \eqref{eq:hermitian-matrix}; here
		$V_{wx}=\overline{V_{x\bar w}}$ as $V$ is real.  This matrix is
		block diagonal, so its
		positivity is equivalent to the second condition in \eqref{eq:three-minors}, together with the
		positive definiteness of the lower-right $2\times2$ block; and a
		Hermitian matrix
		$\left(\begin{smallmatrix}a&b\\ \bar b&\rho\end{smallmatrix}\right)$
		is positive definite if and only if $\rho>0$ and $a\rho-|b|^2>0$.
		Here $\rho=1+V_{w\bar w}$ and
		\[
		a\rho-|b|^2
		=\frac{(\mu+V_{xx})(1+V_{w\bar w})-|\mu+V_{x\bar w}|^2}
		{\mu(1-\mu)}
		=\frac{1+V_{w\bar w}}{\mu(1-\mu)}\,Q[V],
		\]
		so the block is positive definite if and only if $1+V_{w\bar w}>0$
		and $Q[V]>0$.  This proves \eqref{eq:three-minors}.

		At $\xi=0$, our hypothesis for $x\ll 0$ gives $V=V_-(w)$ and
		$\partial u=\mathrm{d}w$.  Thus
		$\omega+\ddc V=\alpha+\mathrm{i}(1+V_{w\bar w})
		\mathrm{d}w\wedge\mathrm{d}\bar w$, so positivity is equivalent to
		$1+V_{w\bar w}>0$.  Finally, where $V$ is constant we have
		$\omega+\ddc V=\omega>0$.
	\end{proof}

	\subsection{The mass identity}

	\begin{lemma}\label{lem:mass}
		For any $V\in C^\infty(X,\mathbb R)$,
		\begin{equation}\label{eq:mass-identity}
			\int_X(\omega+\ddc V)^3
			=\int_X\omega^3+12\pi\int_{\mathbb P^2}
			\left(\int_{\pi^{-1}(y)}|V_w|^2\,\mathrm{d}A_E\right)
			\omega_{\mathrm{FS}}^2(y).
		\end{equation}
	\end{lemma}

	\begin{proof}
		Using Stokes' theorem,
		$\mathrm{d}(\ddc V)=0$, and \eqref{eq:hopf-identities}, one obtains
		\[
		\int_X(\ddc V)^3=0,
		\qquad
		\int_X\omega^2\wedge\ddc V
		=\int_X V\,\ddc(\omega^2)=0.
		\]
		Therefore
		\begin{align*}
			\int_X(\omega+\ddc V)^3-\int_X\omega^3
			&=3\int_X\omega\wedge(\ddc V)^2=3\int_X V\,\ddc\omega\wedge\ddc V\\
			&=-3\int_X V\,\alpha^2\wedge\ddc V=3\int_X\mathrm{i}\partial V\wedge\bar\partial V\wedge\alpha^2.
		\end{align*}
		Because $\alpha^2=\pi^*\omega_{\mathrm{FS}}^2$ already has top degree in the
		base variables, every component of $\partial V$ involving
		$\mathrm{d}\xi_1$ or $\mathrm{d}\xi_2$ vanishes after wedging with
		$\alpha^2$.  On overlaps, the fiber coordinates differ by a
		base-dependent holomorphic function, hence $\partial/\partial w$ and the
		fiber integral in \eqref{eq:mass-identity} are globally well defined.
		Moreover,
		\(
			\mathrm{i}\partial V\wedge\bar\partial V\wedge\alpha^2
			=4\pi |V_w|^2\,\mathrm{d}A_E\wedge\pi^*\omega_{\mathrm{FS}}^2,
		\)
		so Fubini's theorem gives \eqref{eq:mass-identity}.
	\end{proof}
	
	\section{Heat kernel potentials on the elliptic fiber}

	\subsection{The heat kernel and Fourier analysis}

	For $t>0$, define the heat kernel on $E$ by
	\begin{equation}
		k_t(w)\coloneqq \frac{2}{t}\sum_{\lambda\in\Lambda}
		\exp\left(-\frac{|w-\lambda|^2}{t}\right),
		\qquad
		\Lambda:=\mathbb Z+2\pi\mathrm{i}\mathbb Z,
		\label{eq:heat-kernel}
	\end{equation}
	the sum being evaluated at any lift of $w\in E$.  The series and all
	its derivatives converge locally uniformly, and the sum is unchanged
	when $w$ is replaced by $w+\lambda$ with $\lambda\in\Lambda$; thus
	$k_t$ is well defined, smooth, and strictly positive.  Unfolding a
	fundamental rectangle and using \eqref{eq:def-dA},
	\[
	\int_E k_t\,\mathrm{d}A_E
	=\frac{\mathrm{i}}{4\pi}\int_{\mathbb C}\frac2t\,
	\mathrm{e}^{-|w|^2/t}\,\mathrm{d}w\wedge \mathrm{d}\bar w
	=\frac{1}{\pi t}\int_{\mathbb R^2}\mathrm{e}^{-(a^2+b^2)/t}
	\,\mathrm{d}a\,\mathrm{d}b=1 .
	\]
	
	We shall also use Fourier analysis on $E$.  For
	$(p,q)\in\mathbb Z^2$, set
	\[
	\nu_{p,q}=p+\frac{\mathrm{i}q}{2\pi},
	\qquad
	\chi_{p,q}\colon E\to\mathbb C,
	\qquad
	\chi_{p,q}(w)
	=\mathrm{e}^{2\pi\mathrm{i}p\operatorname{Re}w
		+\mathrm{i}q\operatorname{Im}w};
	\]
	each $\chi_{p,q}$ is $\Lambda$-periodic, and
	$\{\chi_{p,q}\}_{(p,q)\in\mathbb Z^2}$ is an orthonormal basis of
	$L^2(E,\mathrm{d}A_E)$. Writing
	$\partial_w=\tfrac12(\partial_a-\mathrm{i}\partial_b)$ and
	$\partial_{\bar w}=\tfrac12(\partial_a+\mathrm{i}\partial_b)$, we
	have
	\begin{equation}
		\partial_w\chi_{p,q}
		=\left(\pi\mathrm{i}p+\frac q2\right)\chi_{p,q},
		\qquad
		\partial_{\bar w}\chi_{p,q}
		=\left(\pi\mathrm{i}p-\frac q2\right)\chi_{p,q},
		\qquad
		(\chi_{p,q})_{w\bar w}=-\pi^2|\nu_{p,q}|^2\,\chi_{p,q}.
		\label{eq:character-derivatives}
	\end{equation}
	The Fourier coefficients of $k_t$ are computed by unfolding:
	\[
	\int_E k_t\,\bar\chi_{p,q}\,\mathrm{d}A_E
	=\frac{1}{\pi t}\int_{\mathbb R^2}
	\mathrm{e}^{-(a^2+b^2)/t}\,
	\mathrm{e}^{-\mathrm{i}(2\pi pa+qb)}\,\mathrm{d}a\,\mathrm{d}b
	=\mathrm{e}^{-t\pi^2p^2}\,\mathrm{e}^{-tq^2/4}
	=\mathrm{e}^{-\pi^2|\nu_{p,q}|^2t}.
	\]
	This yields the Fourier expansion
	of $k_t$, using Poisson summation for the lattice $\Lambda$:
	\begin{equation}
		k_t(w)=\sum_{(p,q)\in\mathbb Z^2}
		\mathrm{e}^{-\pi^2|\nu_{p,q}|^2t}\,\chi_{p,q}(w),
		\label{eq:heat-Fourier}
	\end{equation}
	the series converging in $C^\infty(E)$ for every $t>0$, locally
	uniformly in $t$.  In particular, by
	\eqref{eq:character-derivatives},
	\begin{equation}
		\partial_t k_t=\partial_w\partial_{\bar w}k_t ,
		\label{eq:heat-equation}
	\end{equation}
	and $k_t\to1$ exponentially fast in $C^\infty(E)$ as $t\to\infty$.
	
	\subsection{The fiber potential and its energy}
	
	\begin{lemma}
		\label{lem:heat-fiber}
		For $0<\eta<1/10$, define
		$F_\eta\colon(0,\infty)\times E\to\mathbb R$ by
		\begin{equation}
			F_\eta(t,w)\coloneqq -\eta\int_t^\infty\bigl(k_r(w)-1\bigr)\,\mathrm{d}r
			=-\frac{\eta}{\pi^2}
			\sum_{(p,q)\ne(0,0)}
			\frac{\mathrm{e}^{-\pi^2|\nu_{p,q}|^2t}}
			{|\nu_{p,q}|^2}\,\chi_{p,q}(w).
			\label{eq:heat-potential}
		\end{equation}
		Then $F_\eta$ is smooth and real, and the following hold.
		\begin{enumerate}[label={(\arabic*)}]
			\item One has
			\begin{equation}
				(F_\eta)_t=(F_\eta)_{w\bar w}=\eta(k_t-1),
				\qquad
				1+(F_\eta)_{w\bar w}=1-\eta+\eta k_t>1-\eta.
				\label{eq:heat-identity}
			\end{equation}
			\item One has
			\begin{equation}
				(\log k_t)_{w\bar w}\ge-\frac1t
				\qquad\text{on }E,\quad t>0.
				\label{eq:log-heat}
			\end{equation}
			\item For $0<t\le1$,
			\begin{equation}
				\int_E|(F_\eta)_w(t,w)|^2\,\mathrm{d}A_E
				=2\eta^2\log\frac1t+O(\eta^2).
				\label{eq:heat-energy}
			\end{equation}
			Here the implicit constant is uniform for $0<t\le1$ and
			$0<\eta<1/10$.
			\item $F_\eta$ is exponentially small in the sense that, for all fixed integers $a,m\ge0$ and
			all $t\ge1$:
			\begin{equation}
				\|\partial_t^a F_\eta(t,\cdot)\|_{C^m(E)}
				\le C_{a,m}\,\eta\,\mathrm{e}^{-t/4}.
				\label{eq:heat-flat-F}
			\end{equation}
		\end{enumerate}
	\end{lemma}
	
	\begin{proof}
		Since
		$k_r-1\to0$ exponentially fast, the integral in
		\eqref{eq:heat-potential} converges, and termwise integration in $r$
		gives the stated Fourier series.  The series and all its
		$(t,w)$-derivatives converge uniformly on $[t_0,\infty)\times E$ for
		each $t_0>0$, so $F_\eta$ is smooth; it is real because the
		coefficients are real and invariant under
		$(p,q)\mapsto(-p,-q)$.
		
		(1)  Differentiating \eqref{eq:heat-potential} in $t$ gives
		$(F_\eta)_t=\eta(k_t-1)$; differentiating the Fourier series in
		$w,\bar w$ and using \eqref{eq:character-derivatives} gives
		\[
		(F_\eta)_{w\bar w}
		=\eta\sum_{(p,q)\ne(0,0)}\mathrm{e}^{-\pi^2|\nu_{p,q}|^2t}\chi_{p,q}
		=\eta(k_t-1). 
		\]
		The last inequality in \eqref{eq:heat-identity} holds
		because $k_t>0$.
		
		(2)  Fix a lift $w$ and introduce the probability weights
		\[
		P_{\lambda}\coloneqq \frac{\mathrm{e}^{-|w-\lambda|^2/t}}
		{\sum_{\lambda'\in\Lambda}\mathrm{e}^{-|w-\lambda'|^2/t}},
		\qquad\lambda\in\Lambda .
		\]
		From $\partial_{\bar w}|w-\lambda|^2=w-\lambda$ we get
		\begin{equation}
			t\,\partial_{\bar w}\log k_t(w)
			=-\sum_{\lambda\in\Lambda}(w-\lambda)\,P_\lambda ,
			\label{eq:heat-log-gradient-exact}
		\end{equation}
		and differentiating once more, now in $w$, we find
		\[
		(\log k_t)_{w\bar w}
		=-\frac1t+\frac1{t^2}
		\left\{
		\sum_{\lambda}|w-\lambda|^2P_{\lambda}
		-\Bigl|\sum_{\lambda}(w-\lambda)P_{\lambda}\Bigr|^2
		\right\}.
		\]
		The bracket is the variance of the $\mathbb C$-valued random variable
		$w-\lambda$ under the probability distribution $(P_\lambda)$, hence
		nonnegative.  This proves \eqref{eq:log-heat}.
		
		(3) Parseval's identity applied to \eqref{eq:heat-potential} yields
		\begin{equation}
			\mathcal E_\eta(t)\coloneqq \int_E|(F_\eta)_w|^2\,\mathrm{d}A_E
			=\frac{\eta^2}{\pi^2}
			\sum_{(p,q)\ne(0,0)}
			\frac{\mathrm{e}^{-2\pi^2|\nu_{p,q}|^2t}}
			{|\nu_{p,q}|^2}.
			\label{eq:heat-parseval}
		\end{equation}
		Differentiating termwise and comparing with \eqref{eq:heat-Fourier}, we find
		\[
		\mathcal E_\eta'(t)
		=-2\eta^2\sum_{(p,q)\ne(0,0)}\mathrm{e}^{-2\pi^2|\nu_{p,q}|^2t}
		=-2\eta^2\bigl(k_{2t}(0)-1\bigr).
		\]
		On the other hand, evaluating \eqref{eq:heat-kernel} at $w=0$ and time $2r$, and using $|\lambda|\ge1$ together with the elementary inequality 
		\(
		\mathrm{e}^{-|\lambda|^2/(2r)}\le
		\mathrm{e}^{-1/(4r)}\mathrm{e}^{-|\lambda|^2/4}
		\)
		for $0<r\le1$, we get
		\[
		k_{2r}(0)
		=\frac1r\Bigl(1+
		\sum_{\lambda\in\Lambda\setminus\{0\}}
		\mathrm{e}^{-|\lambda|^2/(2r)}\Bigr)
		=\frac1r+O\left(\frac{\mathrm{e}^{-1/(4r)}}{r}\right),
		\qquad 0<r\le1 .
		\]
		The error $r^{-1}\mathrm{e}^{-1/(4r)}$ is integrable on $(0,1]$, and
		$\mathcal E_\eta(1)\le C\eta^2$ by \eqref{eq:heat-parseval}.
		Integrating $\mathcal E_\eta'$ over $[t,1]$ therefore gives
		\eqref{eq:heat-energy}.
		
		(4)  Fix $a$ and $m$, and abbreviate $\nu\coloneqq\nu_{p,q}$.  Every
		derivative $\partial_t^a\partial^{m'}$ with $m'\le m$, applied to the
		series \eqref{eq:heat-potential}, is a sum over $(p,q)\ne(0,0)$ of terms
		bounded by a polynomial in $|\nu|$ times
		$\eta\mathrm{e}^{-\pi^2|\nu|^2t}$.  Writing
		\[
		\mathrm{e}^{-\pi^2|\nu|^2t}
		=\mathrm{e}^{-t/4}\,\mathrm{e}^{-(\pi^2|\nu|^2-1/4)t}
		\le \mathrm{e}^{-t/4}\,\mathrm{e}^{-(\pi^2|\nu|^2-1/4)}
		\]
		for $t\ge1$, and noting that
		$\sum_{\nu\ne 0}\mathrm{polynomial}(|\nu|)\,
		\mathrm{e}^{-(\pi^2|\nu|^2-1/4)}<\infty$, we obtain
		\eqref{eq:heat-flat-F}.
	\end{proof}
	
	\subsection{A heat-kernel estimate}

	\begin{lemma}
		There is an absolute constant $C_0\ge1$ such that, for
		$0<t\le1$ and $w\in E$,
		\begin{equation}
			tk_t(w)\le C_0\exp\left(
			-\frac{|t\,\partial_{\bar w}\log k_t(w)|^2}{C_0t}
			\right).
			\label{eq:heat-localization}
		\end{equation}
	\end{lemma}

	\begin{proof}
		Fix a lift of $w$, choose a nearest lattice point $\lambda_0$, and put
		\[
			d\coloneqq|w-\lambda_0|,
			\qquad
			s_\lambda\coloneqq|w-\lambda|^2-d^2.
		\]
		The number $d$ is uniformly bounded, $s_\lambda\ge0$, and the triangle
		inequality gives
		\[
			s_\lambda\ge\frac{|\lambda-\lambda_0|^2}{4}-C.
		\]
		For all sufficiently large $|\lambda-\lambda_0|$, the right-hand side
		is at least $|\lambda-\lambda_0|^2/8$; the remaining lattice points
		form a fixed finite set.  Since $s_\lambda\ge0$ and $d$ is bounded,
		there is an absolute constant $C_0\ge1$ such that, uniformly for
		$0<t\le1$,
		\[
			\sum_{\lambda\in\Lambda}\mathrm{e}^{-s_\lambda/t}
			+\sum_{\lambda\in\Lambda}|w-\lambda|
			\mathrm{e}^{-s_\lambda/t}\le C_0.
		\]
		Consequently
		\[
			tk_t(w)=2\,\mathrm{e}^{-d^2/t}
			\sum_{\lambda\in\Lambda}\mathrm{e}^{-s_\lambda/t}
			\le C_0\,\mathrm{e}^{-d^2/t}.
		\]
		For the gradient, the probability weights used in
		\eqref{eq:heat-log-gradient-exact} satisfy
		$P_\lambda\le\mathrm{e}^{-s_\lambda/t}$, and hence
		\[
			\bigl|t\,\partial_{\bar w}\log k_t\bigr|
			\le d+\sum_{\lambda\ne\lambda_0}|w-\lambda|
			\mathrm{e}^{-s_\lambda/t}.
		\]
		If $d\ge1/4$, the preceding uniform bound makes the right-hand side
		at most $C_0d$.  If $d<1/4$, then, for $\lambda\ne\lambda_0$,
		\(
			s_\lambda\ge|\lambda-\lambda_0|
			\bigl(|\lambda-\lambda_0|-2d\bigr)
			\ge\frac{|\lambda-\lambda_0|}{2},
		\)
		and summation over the lattice gives
		$\sum_{\lambda\ne\lambda_0}|w-\lambda|
		\mathrm{e}^{-s_\lambda/t}\le
		C_0\mathrm{e}^{-1/(C_0t)}$.  Thus, in both cases,
		\[
			G\coloneqq\bigl|t\,\partial_{\bar w}\log k_t(w)\bigr|
			\le C_0\left(d+\mathrm{e}^{-1/(C_0t)}\right).
		\]
		If $G>2C_0\mathrm{e}^{-1/(C_0t)}$, then $d\ge G/(2C_0)$,
		and the estimate for $tk_t$ proves \eqref{eq:heat-localization}.
		Otherwise $G^2/t$ is uniformly bounded for $0<t\le1$, and the same
		conclusion follows from $tk_t\le C_0$, after increasing the absolute
		constant.
	\end{proof}

	\begin{lemma}
		\label{lem:heat-transition} There exists $C>0$ such that,
		for $0<\kappa<\frac{1}{200C_0}$,
		\(
			0<\eta<\frac{1}{10},\ 
			0<t\le1, \frac{1}{2}\le\mu\le1,
		\)
and
		\[
			\sigma\coloneqq
			\sqrt{\frac{2\kappa}{\eta}\log\frac1{\eta t}},\qquad
			h\coloneqq\eta k_t,\qquad
			\zeta\coloneqq\frac{\sigma}{\mu}\,
			t\,\partial_{\bar w}\log k_t,
		\]
		we have, pointwise on $E$,
		\begin{equation}\label{eq:heat-transition}
			\kappa tk_t-
			\frac{\mu^2(1-\eta)h}{1-\eta+h}
			\left|\zeta-\frac1{1-\eta}\right|^2
			\le C\kappa\exp\left(-\frac{1}{C\sigma^2t}\right).
		\end{equation}
	\end{lemma}

	\begin{proof}
		Let $C_0\ge1$ be the constant in \eqref{eq:heat-localization},
		and take any $\kappa>0$ with $C_0\kappa<1/200$.
		If $\sigma^2t\ge1/100$, the left-hand side of
		\eqref{eq:heat-transition} is at most $\kappa tk_t\le C_0\kappa$.
		Increasing $C$ gives the assertion.

		Suppose next that $\sigma^2t<1/100$ and
		$|\zeta-(1-\eta)^{-1}|\ge1/4$.  Since $\mu\ge1/2$ and
		$\eta<1/10$, the subtracted square is at least
		\(
			\frac{h}{100(1+h)}.
		\)
		Moreover, $\log(1/(\eta t))>2$, $tk_t\le C_0$, and the definition
		of $\sigma$ give
		\[
			\kappa t(1+h)
			\le\frac{\eta}{400}+C_0\kappa\eta
			\le\frac{3\eta}{400}<\frac{\eta}{100}
		\]
		by the choice of $\kappa$.  Since $h=\eta k_t$, this is
		equivalent to $\kappa tk_t\le h/(100(1+h))$; hence the left-hand
		side of \eqref{eq:heat-transition} is nonpositive.

		It remains to consider $\sigma^2t<1/100$ and
		$|\zeta-(1-\eta)^{-1}|<1/4$.  Then $|\zeta|\ge3/4$, so
		\[
			\bigl|t\,\partial_{\bar w}\log k_t\bigr|
			=\frac{\mu|\zeta|}{\sigma}\ge\frac{3}{8\sigma}.
		\]
		Hence \eqref{eq:heat-localization} yields
		\(
			tk_t\le C\exp\left(-\frac1{C\sigma^2t}\right),
		\)
	which implies
		\eqref{eq:heat-transition}.
	\end{proof}

	\section{Proof of the main theorem}
	
	We construct positive test potentials with unbounded
	fiberwise energy.

	\begin{proposition}\label{prop:transport}
		For every sufficiently small $\eta>0$ there is a smooth function
		$V_\eta\colon\mathbb R\times E\to\mathbb R$ with the following
		properties.
		\begin{enumerate}[label={(\arabic*)},ref={(\arabic*)}]
			\item\label{it:V-ends}
			$V_\eta$ is independent of $x$ for $x\ll0$ and is constant in
			both variables for $x\gg0$.
			\item\label{it:V-positivity}
			One has
			\[
				1+(V_\eta)_{w\bar w}>0,
				\qquad
				\mu+(V_\eta)_x>0,
				\qquad
				Q[V_\eta]>0.
			\]
			\item\label{it:V-energy}
			On the base ball
			$B\coloneqq\{\xi\in\mathbb C^2:|\xi|^2<\mathrm{e}^{-3}\}\Subset U_0$,
			\[
				\inf_{y\in B}
				\int_{\pi^{-1}(y)} |(V_\eta)_w|^2\,\mathrm{d}A_E
				\longrightarrow +\infty
				\qquad\text{as }\eta\downarrow0 .
			\]
		\end{enumerate}
	\end{proposition}

	\begin{proof}[Proof of \cref{thm:main} assuming \cref{prop:transport}]
		Take $V_\eta$ furnished by the proposition.
		Property~\ref{it:V-ends} and \cref{lem:positivity} extend it smoothly
		to $X$.  Property~\ref{it:V-positivity} gives
		$\omega+\ddc V_\eta>0$.  The mass identity \eqref{eq:mass-identity}
		of \cref{lem:mass} together with property~\ref{it:V-energy} then
		gives
		\[
		\int_X(\omega+\ddc V_\eta)^3
		\ge
		\int_X\omega^3
		+12\pi\int_B
		\left(
		\int_{\pi^{-1}(y)}|(V_\eta)_w|^2\,\mathrm{d}A_E
		\right)\omega_{\mathrm{FS}}^2(y)
		\longrightarrow+\infty
		\]
		as $\eta\downarrow0$.  Taking $\varphi_j\coloneqq V_{\eta_j}$ for any
		sequence $\eta_j\downarrow0$ completes the proof.
	\end{proof}
	\begin{proof}[Proof of \cref{prop:transport}]
		We divide the construction into four steps.
		Fix $\kappa>0$ sufficiently small as in \cref{lem:heat-transition}.
		Choose once and for all $\eta_0\in(0,1/10)$ sufficiently small that
		all inequalities below that require $\eta$ to be small hold whenever
		$0<\eta<\eta_0$, and fix such an $\eta$.

		\emph{Step 1.}
		Set
		\[
			\varepsilon_\eta\coloneqq\eta^{1/\eta^2},
			\qquad
			q_0\coloneqq\sqrt{\log\frac1{\eta\varepsilon_\eta}},
			\qquad
			a_\eta\coloneqq\sqrt{\frac{\kappa}{2\eta}}.
		\]
		Choose a smooth nondecreasing function
		$\chi\colon\mathbb R\to[0,1]$ which is zero on $(-\infty,0]$,
		one on $[1,\infty)$, and satisfies $\int_0^1\chi=1/2$.  Define
		\[
			q(x)\coloneqq q_0-\int_0^{a_\eta x}\chi(s)\,\mathrm{d}s,
			\qquad
			t(x)\coloneqq\eta^{-1}\mathrm{e}^{-q(x)^2}.
		\]
		Since $q_x=-a_\eta\chi(a_\eta x)$, one has $q=q_0$ and
		$t=\varepsilon_\eta$ on $x\le0$, whereas
		$q(x)=q_0+1/2-a_\eta x$ for $x\ge a_\eta^{-1}$.  Define
		\[
			b_\eta\coloneqq\frac{q_0-1/2}{a_\eta},\qquad
			x_\eta^*\coloneqq\frac{q_0+1/2}{a_\eta},
		\]
		so that
		\[
			0<a_\eta^{-1}\ll b_\eta<x_\eta^*,\qquad
			q(b_\eta)=1,\qquad q(x_\eta^*)=0,\qquad
			b_\eta,x_\eta^*
			=O\left(\sqrt{\frac{\log(1/\eta)}{\eta}}\right).
		\]
		The function $q$ is shown in \cref{fig:q}.

		\begin{figure}[t]
			\centering
			\begin{tikzpicture}[scale=0.9,line cap=round]
				\pgfmathsetmacro{\xzero}{2.66}
				\pgfmathsetmacro{\xone}{7.60}
				\pgfmathsetmacro{\yzero}{2.60}
				\pgfmathsetmacro{\mright}{-\yzero/(\xone-\xzero)}
				\pgfmathsetmacro{\bpoint}{\xzero+(\yzero-0.55)/(-\mright)}
				\pgfmathsetmacro{\yend}{\mright*(8.4-\xone)}

				\draw[->] (-1.4,0) -- (9.4,0) node[right] {$x$};
				\draw[->] (0,-0.55) -- (0,3.9) node[above] {$q(x)$};
				\draw[densely dashed,gray] (-1.2,3.3) -- (0,3.3);
				\draw[densely dashed,gray] (\xzero,2.6) -- (0,2.6);
				\draw[densely dashed,gray] (\bpoint,0.55) -- (0,0.55);
				\node[left] at (-1.2,3.3) {$q_0$};
				\node[left] at (0,2.6) {$q_0-\tfrac12$};
				\node[left] at (0,0.55) {$1$};

				\draw[thick]
					(-1.2,3.3) -- (0,3.3)
					plot[domain=0:\xzero,samples=180,smooth]
					(\x,{3.3+\mright*\xzero*(
						7*(\x/\xzero)^5-14*(\x/\xzero)^6
						+10*(\x/\xzero)^7-2.5*(\x/\xzero)^8)})
					-- (8.4,\yend);

				\foreach \x/\ell in {
					\xzero/{a_\eta^{-1}},
					\bpoint/{b_\eta}
				}{
					\draw (\x,0.06) -- (\x,-0.06) node[below] {$\ell$};
				}
				\draw (\xone,0.06) -- (\xone,-0.06);
				\node[above] at (\xone,0.12) {$x_\eta^*$};
				\node[below left] at (0,0) {$0$};

				\draw[decorate,decoration={brace,mirror,raise=26pt}]
					(\xzero,0) -- (\bpoint,0)
					node[midway,below=28pt] {$q_x=-a_\eta$};
				\draw[decorate,decoration={brace,mirror,raise=26pt}]
					(\bpoint,0) -- (\xone,0)
					node[midway,below=28pt] {$0\le q\le1$};
			\end{tikzpicture}
			\caption{The graph of $q$.}
			\label{fig:q}
		\end{figure}

		Choose a smooth $\theta\colon\mathbb R\to[0,1]$ which is zero on
		$(-\infty,0]$ and one on $[1,\infty)$, and set
		\[
			W_\eta(x,w)\coloneqq\theta\bigl(q(x)\bigr)F_\eta(t(x),w).
		\]
		The cutoffs are constant on one-sided neighborhoods of their
		endpoints.  Hence $W_\eta$ is smooth, equals
		$F_\eta(\varepsilon_\eta,\cdot)$ on $x\le0$, and vanishes on
		$x\ge x_\eta^*$.  By \eqref{eq:heat-identity},
		\begin{equation}\label{eq:direct-fiber-positive}
			1+(W_\eta)_{w\bar w}
			=1+\theta(q)\eta(k_t-1)>1-\eta.
		\end{equation}

		\emph{Step 2.}
		For $a\in\mathbb R$, write $[a]_-\coloneqq\max\{-a,0\}$.
		We first estimate $[Q[W_\eta]]_-$ on $[0,b_\eta]$, where
		$q\ge1$ and hence $\theta(q)=1$.  Write
		\[
			h\coloneqq\eta k_t,\qquad
			\sigma\coloneqq(\log t)_x,\qquad
			\zeta\coloneqq\frac{\sigma}{\mu}\,
			t\,\partial_{\bar w}\log k_t .
		\]
		For $F=F_\eta(t(x),w)$, the chain rule gives
		\[
			W_x=\sigma tF_t,\qquad
			W_{xx}=(\sigma_x+\sigma^2)tF_t+\sigma^2t^2F_{tt},
			\qquad
			W_{x\bar w}=\sigma tF_{t\bar w}.
		\]
		Since $F_t=h-\eta$, the heat equation
		\eqref{eq:heat-equation} gives
		\[
			F_{t\bar w}=h\,\partial_{\bar w}\log k_t,\qquad
			h_t=h\left((\log k_t)_{w\bar w}
			+|\partial_{\bar w}\log k_t|^2\right).
		\]
		Substitution in $Q$ and completion of the square yield the exact
		identity
		\begin{align}
			Q[W_\eta]={}&
			\sigma^2t\Bigl(
			h\left[1+t(\log k_t)_{w\bar w}\right]-\eta
			\Bigr)
			+\sigma_x tF_t+\mu(1-\mu)-\frac{\mu^2\eta}{1-\eta}
			\nonumber\\
			&+
			\frac{\mu^2(1-\eta)h}{1-\eta+h}
			\left|\zeta-\frac1{1-\eta}\right|^2 .
			\label{eq:direct-Q-identity}
		\end{align}

		For $0\le x\le a_\eta^{-1}$, differentiation gives
		\[
			\sigma=2a_\eta q\chi(a_\eta x),\qquad
			\frac{\eta\sigma^2}{2q^2}=\kappa\chi(a_\eta x)^2,\qquad
			\sigma_x=-\frac{\kappa}{\eta}\chi(a_\eta x)^2
			+\frac{\kappa}{\eta}q\chi'(a_\eta x).
		\]
		Here $q_0-1/2\le q\le q_0$, so
		$t\le\varepsilon_\eta\mathrm{e}^{q_0}=o(\eta^N)$ for every $N>0$.
		Since $\chi'$ is bounded, our choice of $\eta_0$ ensures that
		\[
			\kappa q\chi'(a_\eta x)t\le\eta,\qquad
			\sigma^2t\le\eta^2.
		\]
		In particular, $t<1$.
		If $\chi(a_\eta x)>0$, then
		$0<\kappa\chi(a_\eta x)^2\le\kappa$ and $1/2\le\mu<1$.
		Since $t<1$ and $q^2=\log(1/(\eta t))$,
		\cref{lem:heat-transition} applies with
		$\kappa\chi(a_\eta x)^2$ in place of $\kappa$.
		The term containing $q\chi'(a_\eta x)$ in $\sigma_x tF_t$ is at
		least $-\kappa q\chi'(a_\eta x)t\ge-\eta$, as $F_t\ge-\eta$.
		Since
		$1+t(\log k_t)_{w\bar w}\ge0$ and
		$\mu(1-\mu)-\mu^2\eta/(1-\eta)\ge-2\eta$,
		\eqref{eq:direct-Q-identity}, \eqref{eq:heat-transition}, and
		$\sigma^2t\le\eta^2$ give $[Q[W_\eta]]_-\le C\eta$.
		If $\chi(a_\eta x)=0$, smooth nonnegativity gives
		$\chi(a_\eta x)=\chi'(a_\eta x)=0$, hence
		$\sigma=\sigma_x=0$.  Moreover
		$x\le a_\eta^{-1}=o(1)$, so $\mu<1-\eta$, and we obtain that
		\[
			Q[W_\eta]=\mu-\frac{\mu^2}{1-\eta+h}>0.
		\]
		For $a_\eta^{-1}\le x\le b_\eta$,
		\[
			\sigma=\sqrt{\frac{2\kappa q^2}{\eta}}=2a_\eta q,\qquad
			\sigma_x=-\frac{\kappa}{\eta},\qquad
			\mathrm{d}x=
			\sqrt{\frac{\eta}{2\kappa}}\,
			\frac{\mathrm{d}t}{tq}.
		\]
		When $0<t\le1$, \cref{lem:heat-transition} applies.  Then \eqref{eq:log-heat}, \eqref{eq:heat-transition} and \eqref{eq:direct-Q-identity} give
		\[
			\Bigl[Q[W_\eta]\Bigr]_-
			\le2\eta+2\kappa q^2t
			+C\kappa
			\exp\left(-\frac{\eta}{C\kappa q^2t}\right).
		\]
		Indeed, the first term of \eqref{eq:direct-Q-identity} is at least
		$-\eta\sigma^2t=-2\kappa q^2t$, while
		\[
			\sigma_x tF_t=-\kappa t(k_t-1)
			=-\kappa tk_t+\kappa t,
		\]
		and the completed square absorbs $\kappa tk_t$ up to the displayed
		exponential error.

		For $t\ge1$, one has $1\le q^2\le\log(1/\eta)$.  Applying the
		chain rule and \eqref{eq:heat-flat-F} directly to $Q$, and using
		$1-\eta+h\ge1-\eta$, gives
		\[
			\Bigl[Q[W_\eta]\Bigr]_-
			\le C\kappa q^2t^2\mathrm{e}^{-t/4}.
		\]

		Let
		\[
			d_\eta(x)\coloneqq
			\max_{w\in E}\Bigl[Q[W_\eta](x,w)\Bigr]_-.
		\]
		We integrate these bounds using $q_x=-a_\eta$ on
		$[a_\eta^{-1},b_\eta]$.  On the part where $t\le1$, one has
		$q\ge\sqrt{\log(1/\eta)}$, $t=\eta^{-1}\mathrm{e}^{-q^2}$, and
		$\mathrm{d}x=-\mathrm{d}q/a_\eta$.  Thus the term $2\kappa q^2t$
		has integral at most
		\[
			\frac{2\kappa}{a_\eta\eta}
			\int_{\sqrt{\log(1/\eta)}}^\infty
			q^2\mathrm{e}^{-q^2}\,\mathrm{d}q
			=o(1).
		\]
		Here we used
		$\int_R^\infty q^2\mathrm{e}^{-q^2}\,\mathrm{d}q
		\le CR\mathrm{e}^{-R^2}$ for $R\ge1$.
		For the exponential term, the same change of variables gives
		\[
			\frac{C\kappa}{a_\eta}
			\int_{\sqrt{\log(1/\eta)}}^{q_0}
			\exp\left(-\frac{\eta^2\mathrm{e}^{q^2}}
			{C\kappa q^2}\right)\mathrm{d}q.
		\]
		Split this integral at $2\sqrt{\log(1/\eta)}$.  The first part has
		length $O(\sqrt{\log(1/\eta)})$.  On the second part,
		$\eta^2\mathrm{e}^{q^2}/q^2\ge\mathrm{e}^{q^2/4}$ for sufficiently
		small $\eta$.  Extending the upper limit to infinity shows that the
		remaining integral is uniformly bounded.
		Its contribution is therefore $o(1)$.

		On the part where $t\ge1$, one has
		$1\le q\le\sqrt{\log(1/\eta)}$ and
		$\mathrm{d}x=\sqrt{\eta/(2\kappa)}\,\mathrm{d}t/(tq)$.  Hence
		\[
			\int \Bigl[Q[W_\eta]\Bigr]_-\,\mathrm{d}x
			\le C\sqrt\eta
			\int_1^{1/(\mathrm{e}\eta)}
			q t\mathrm{e}^{-t/4}\,\mathrm{d}t
			=o(1).
		\]
		Here $q\le\sqrt{\log(1/\eta)}$, while
		$t\mathrm{e}^{-t/4}$ is integrable on $[1,\infty)$.
		The constant term and the interval $[0,a_\eta^{-1}]$ together
		contribute at most $C\eta b_\eta=o(1)$.  Hence
		$\int_0^{b_\eta}d_\eta(x)\,\mathrm{d}x=o(1)$.

		The same choice of $t(x)$ controls the negative slope.  If $t\le1$,
		then
		\[
			[W_x]_-\le\sigma\eta t
			\le\sqrt{2\kappa\eta}\,
			t\sqrt{\log\frac1{\eta t}}
			\le C\sqrt{\eta\log\frac1\eta}=o(1).
		\]
		If $t\ge1$ and $q\ge1$, the chain rule and
		\eqref{eq:heat-flat-F} give the same bound.  If
		$b_\eta\le x\le x_\eta^*$, then $0\le q\le1$ and
		$(\mathrm{e}\eta)^{-1}\le t\le\eta^{-1}$.  Direct differentiation
		and \eqref{eq:heat-flat-F} give
		\[
			\sup_{w\in E}\biggl(
			|W_x|+\Bigl|Q[W_\eta]-Q[0]\Bigr|
			\biggr)
			\le C\mathrm{e}^{-1/(5\mathrm{e}\eta)}.
		\]
		Consequently
		\begin{equation}\label{eq:direct-negative-slope}
			\sup_{\mathbb R\times E}[W_x]_-
			=o(1).
		\end{equation}
		For $b_\eta\le x\le x_\eta^*$, we have
		\(
			\mu(1-\mu)\ge\frac14\mathrm{e}^{-x_\eta^*}
			\text{ and }
			C\mathrm{e}^{-1/(5\mathrm{e}\eta)}
			=o(\mathrm{e}^{-x_\eta^*}),
		\)
		by the estimate for $x_\eta^*$ in Step~1.  Together with
		\eqref{eq:direct-fiber-positive}, comparison with
		$Q[0]=\mu(1-\mu)$ gives $Q[W_\eta]>0$ there.  For $x\le0$,
		\[
			Q[W_\eta]
			=\mu-\frac{\mu^2}{1+(W_\eta)_{w\bar w}}
			\ge\mu\left(1-\frac{\mu}{1-\eta}\right)>0,
		\]
		and for $x\ge x_\eta^*$ one has
		$Q[W_\eta]=\mu(1-\mu)>0$.  Thus $d_\eta$ is continuous,
		supported in $[0,b_\eta]$, and
		\begin{equation}\label{eq:direct-negative-mass}
			\int_{\mathbb R}d_\eta(x)\,\mathrm{d}x
			=o(1).
		\end{equation}
		On $[-2,-1]$, where
		$1/9\le\mu\le3/10$, the formula for $x\le0$ gives
		\[
			Q[W_\eta]\ge\frac2{27}.
		\]

		\emph{Step 3.}
		We now add a function of $x$ to make $Q$ positive.
		Take a smooth nonnegative function $r_\eta$, supported in
		$(-1/2,b_\eta+1/2)$, such that
		\[
			r_\eta>d_\eta\quad\text{on }[0,b_\eta],
			\qquad
			m_\eta\coloneqq\int_{\mathbb R}r_\eta\,\mathrm{d}x
			\le\int_{\mathbb R}d_\eta\,\mathrm{d}x+\eta.
		\]
		By \eqref{eq:direct-negative-mass}, $m_\eta=o(1)$.
		Fix a smooth $\psi\ge0$, supported in $(-2,-1)$, with
		$\int_{\mathbb R}\psi\,\mathrm{d}x=1$, and define
		\[
			g_\eta(x)\coloneqq
			\int_{-\infty}^x(x-s)
			\Bigl(r_\eta(s)-m_\eta\psi(s)\Bigr)\,\mathrm{d}s.
		\]
		Thus $g_\eta''=r_\eta-m_\eta\psi$.  Since $\psi$ is supported to
		the left of $r_\eta$ and the two terms have the same integral,
		\[
			-m_\eta\le g_\eta'\le0,\qquad
			g_\eta'=0\quad\text{for }x\le-2
			\quad\text{and for }x\ge b_\eta+\frac12.
		\]
		In particular, $g_\eta=0$ at the negative end and is constant at
		the positive end.

		Set $V_\eta\coloneqq W_\eta+g_\eta$.  Since $g_\eta$ depends only
		on $x$,
		\[
			Q[V_\eta]=Q[W_\eta]+r_\eta-m_\eta\psi.
		\]
		On $[0,b_\eta]$ this is positive because $r_\eta>d_\eta$ and
		$\psi=0$.  On $(-2,-1)$ it is at least
		$2/27-m_\eta\|\psi\|_\infty>0$.  Elsewhere $Q[W_\eta]>0$,
		$\psi=0$, and $r_\eta\ge0$.  Hence $Q[V_\eta]>0$ everywhere.
		For the second inequality of \ref{it:V-positivity}, both $W_x$ and $g_\eta'$ vanish on $x\le-2$.
		On $[-2,0]$ one has $W_x=0$, $\mu\ge\mu(-2)>0$, and
		$g_\eta'\ge-m_\eta$; on $[0,\infty)$ one has $\mu\ge1/2$,
		$W_x\ge-o(1)$ by
		\eqref{eq:direct-negative-slope}, and $g_\eta'\ge-m_\eta$.
		Thus $\mu+(V_\eta)_x>0$.  Finally,
		\eqref{eq:direct-fiber-positive} gives
		\(
			1+(V_\eta)_{w\bar w}>1-\eta>0,
		\)
        proving \ref{it:V-positivity}.
		The endpoint descriptions of $W_\eta$ and $g_\eta$ prove
		property~\ref{it:V-ends}.

		\emph{Step 4.}
		It remains to check the energy at the center of the base chart.
		Over $0<|\xi|^2<\mathrm{e}^{-3}$ one has $x<-3$, hence
		$g_\eta=0$ and
		$V_\eta=F_\eta(\varepsilon_\eta,\cdot)$.  The same equality holds
		over $\xi=0$ by the smooth negative-end extension.  Since
		$0<\varepsilon_\eta<1$, \eqref{eq:heat-energy} gives, uniformly
		for $y\in B$,
		\[
			\int_{\pi^{-1}(y)}|(V_\eta)_w|^2\,\mathrm{d}A_E
			=2\eta^2\log\frac1{\varepsilon_\eta}+O(\eta^2)
			=2\log\frac1\eta+O(\eta^2)
			\longrightarrow+\infty.
		\]
		This proves \cref{prop:transport}.
	\end{proof}
	\printbibliography

@article {BT82,
    AUTHOR = {Bedford, Eric and Taylor, B. A.},
     TITLE = {A new capacity for plurisubharmonic functions},
   JOURNAL = {Acta Math.},
  FJOURNAL = {Acta Mathematica},
    VOLUME = {149},
      YEAR = {1982},
    NUMBER = {1-2},
     PAGES = {1--40},
      ISSN = {0001-5962,1871-2509},
   MRCLASS = {32F05 (31C10 32C30)},
  MRNUMBER = {674165},
MRREVIEWER = {Guy\ Laville},
       DOI = {10.1007/BF02392348},
       URL = {https://doi.org/10.1007/BF02392348},
}

@article {Y78,
    AUTHOR = {Yau, Shing Tung},
     TITLE = {On the {R}icci curvature of a compact {K}\"{a}hler manifold
              and the complex {M}onge-{A}mp\`ere equation. {I}},
   JOURNAL = {Comm. Pure Appl. Math.},
  FJOURNAL = {Communications on Pure and Applied Mathematics},
    VOLUME = {31},
      YEAR = {1978},
    NUMBER = {3},
     PAGES = {339--411},
      ISSN = {0010-3640,1097-0312},
   MRCLASS = {53C55 (32C10 35J60)},
  MRNUMBER = {480350},
MRREVIEWER = {Robert\ E.\ Greene},
       DOI = {10.1002/cpa.3160310304},
       URL = {https://doi.org/10.1002/cpa.3160310304},
}

@incollection {Cal57,
    AUTHOR = {Calabi, Eugenio},
     TITLE = {On {K}\"{a}hler manifolds with vanishing canonical class},
 BOOKTITLE = {Algebraic geometry and topology. {A} symposium in honor of
              {S}. {L}efschetz},
     PAGES = {78--89},
 PUBLISHER = {Princeton Univ. Press, Princeton, NJ},
      YEAR = {1957},
   MRCLASS = {53.3X},
  MRNUMBER = {85583},
MRREVIEWER = {P.\ Dolbeault},
}

@book {GZ17,
    AUTHOR = {Guedj, Vincent and Zeriahi, Ahmed},
     TITLE = {Degenerate complex {M}onge-{A}mp\`ere equations},
    SERIES = {EMS Tracts in Mathematics},
    VOLUME = {26},
 PUBLISHER = {European Mathematical Society (EMS), Z\"{u}rich},
      YEAR = {2017},
     PAGES = {xxiv+472},
      ISBN = {978-3-03719-167-5},
   MRCLASS = {32W20 (32Q20 32U15 32U20 32U40 35J96)},
  MRNUMBER = {3617346},
MRREVIEWER = {Slimane\ Benelkourchi},
       DOI = {10.4171/167},
       URL = {https://doi.org/10.4171/167},
}

@article{BGL,
  author        = {Boucksom, S\'ebastien and Guedj, Vincent and Lu, Chinh H.},
  title         = {Volumes of {Bott--Chern} classes},
  journal       = {Peking Math. J.},
  year          = {2025},
  doi           = {10.1007/s42543-025-00105-2},
  url           = {https://doi.org/10.1007/s42543-025-00105-2},
  eprint        = {2406.01090},
  archivePrefix = {arXiv},
  primaryClass  = {math.CV},
}

@article {GL22,
    AUTHOR = {Guedj, Vincent and Lu, Chinh H.},
     TITLE = {Quasi-plurisubharmonic envelopes 2: {B}ounds on
              {M}onge-{A}mp\`ere volumes},
   JOURNAL = {Algebr. Geom.},
  FJOURNAL = {Algebraic Geometry},
    VOLUME = {9},
      YEAR = {2022},
    NUMBER = {6},
     PAGES = {688--713},
      ISSN = {2313-1691,2214-2584},
   MRCLASS = {32W20 (32J18 32U05 35A23 35J96)},
  MRNUMBER = {4518244},
MRREVIEWER = {S\l awomir\ Dinew},
       DOI = {10.14231/ag-2022-021},
       URL = {https://doi.org/10.14231/ag-2022-021},
}

@article {GL23,
    AUTHOR = {Guedj, Vincent and Lu, Chinh H.},
     TITLE = {Quasi-plurisubharmonic envelopes 3: {S}olving
              {M}onge-{A}mp\`ere equations on hermitian manifolds},
   JOURNAL = {J. Reine Angew. Math.},
  FJOURNAL = {Journal f\"{u}r die Reine und Angewandte Mathematik. [Crelle's
              Journal]},
    VOLUME = {800},
      YEAR = {2023},
     PAGES = {259--298},
      ISSN = {0075-4102,1435-5345},
   MRCLASS = {32W20},
  MRNUMBER = {4609828},
MRREVIEWER = {S\l awomir\ Dinew},
       DOI = {10.1515/crelle-2023-0030},
       URL = {https://doi.org/10.1515/crelle-2023-0030},
}

@article {Che87,
    AUTHOR = {Cherrier, Pascal},
     TITLE = {\'{E}quations de {M}onge-{A}mp\`ere sur les vari\'{e}t\'{e}s
              hermitiennes compactes},
   JOURNAL = {Bull. Sci. Math. (2)},
  FJOURNAL = {Bulletin des Sciences Math\'{e}matiques. 2e S\'{e}rie},
    VOLUME = {111},
      YEAR = {1987},
    NUMBER = {4},
     PAGES = {343--385},
      ISSN = {0007-4497},
   MRCLASS = {58G30 (32C10 35J60 53C55)},
  MRNUMBER = {921559},
MRREVIEWER = {John\ M.\ Lee},
}

@article {TW10,
    AUTHOR = {Tosatti, Valentino and Weinkove, Ben},
     TITLE = {The complex {M}onge-{A}mp\`ere equation on compact {H}ermitian
              manifolds},
   JOURNAL = {J. Amer. Math. Soc.},
  FJOURNAL = {Journal of the American Mathematical Society},
    VOLUME = {23},
      YEAR = {2010},
    NUMBER = {4},
     PAGES = {1187--1195},
      ISSN = {0894-0347,1088-6834},
   MRCLASS = {32W20 (32Q15)},
  MRNUMBER = {2669712},
MRREVIEWER = {S\l awomir\ Ko\l odziej},
       DOI = {10.1090/S0894-0347-2010-00673-X},
       URL = {https://doi.org/10.1090/S0894-0347-2010-00673-X},
}

@incollection {DK12,
    AUTHOR = {Dinew, S\l awomir and Ko\l odziej, S\l awomir},
     TITLE = {Pluripotential estimates on compact {H}ermitian manifolds},
 BOOKTITLE = {Advances in geometric analysis},
    SERIES = {Adv. Lect. Math. (ALM)},
    VOLUME = {21},
     PAGES = {69--86},
 PUBLISHER = {Int. Press, Somerville, MA},
      YEAR = {2012},
      ISBN = {978-1-57146-248-0},
   MRCLASS = {32U05 (32U40 32W20 53C55)},
  MRNUMBER = {3077248},
MRREVIEWER = {Slimane\ Benelkourchi},
}

@incollection {KN15,
    AUTHOR = {Ko\l odziej, S\l awomir and Nguyen, Ngoc Cuong},
     TITLE = {Weak solutions to the complex {M}onge-{A}mp\`ere equation on
              {H}ermitian manifolds},
 BOOKTITLE = {Analysis, complex geometry, and mathematical physics: in honor
              of {D}uong {H}. {P}hong},
    SERIES = {Contemp. Math.},
    VOLUME = {644},
     PAGES = {141--158},
 PUBLISHER = {Amer. Math. Soc., Providence, RI},
      YEAR = {2015},
      ISBN = {978-1-4704-1464-1},
   MRCLASS = {32W20 (32U40)},
  MRNUMBER = {3372464},
MRREVIEWER = {Ly\ Kim\ Ha},
       DOI = {10.1090/conm/644/12775},
       URL = {https://doi.org/10.1090/conm/644/12775},
}

@article {Kol98,
    AUTHOR = {Ko\l odziej, S\l awomir},
     TITLE = {The complex {M}onge-{A}mp\`ere equation},
   JOURNAL = {Acta Math.},
  FJOURNAL = {Acta Mathematica},
    VOLUME = {180},
      YEAR = {1998},
    NUMBER = {1},
     PAGES = {69--117},
      ISSN = {0001-5962,1871-2509},
   MRCLASS = {32F07 (32C17 35J60)},
  MRNUMBER = {1618325},
MRREVIEWER = {M.\ Klimek},
       DOI = {10.1007/BF02392879},
       URL = {https://doi.org/10.1007/BF02392879},
}

@article {Gau77,
    AUTHOR = {Gauduchon, Paul},
     TITLE = {Le th\'{e}or\`eme de l'excentricit\'{e} nulle},
   JOURNAL = {C. R. Acad. Sci. Paris S\'{e}r. A-B},
  FJOURNAL = {Comptes Rendus Hebdomadaires des S\'{e}ances de l'Acad\'{e}mie
              des Sciences. S\'{e}ries A et B},
    VOLUME = {285},
      YEAR = {1977},
    NUMBER = {5},
     PAGES = {A387--A390},
      ISSN = {0151-0509},
   MRCLASS = {53C55 (32L05)},
  MRNUMBER = {470920},
}

@incollection {Hopf48,
    AUTHOR = {Hopf, H.},
     TITLE = {Zur {T}opologie der komplexen {M}annigfaltigkeiten},
 BOOKTITLE = {Studies and {E}ssays {P}resented to {R}. {C}ourant on his 60th
              {B}irthday, {J}anuary 8, 1948},
     PAGES = {167--185},
 PUBLISHER = {Interscience Publishers, New York},
      YEAR = {1948},
   MRCLASS = {56.0X},
  MRNUMBER = {23054},
MRREVIEWER = {S.\ Eilenberg},
}

@article {Chi14,
    AUTHOR = {Chiose, Ionu\c{t}},
     TITLE = {Obstructions to the existence of {K}\"{a}hler structures on
              compact complex manifolds},
   JOURNAL = {Proc. Amer. Math. Soc.},
  FJOURNAL = {Proceedings of the American Mathematical Society},
    VOLUME = {142},
      YEAR = {2014},
    NUMBER = {10},
     PAGES = {3561--3568},
      ISSN = {0002-9939,1088-6826},
   MRCLASS = {32J27 (32Q15)},
  MRNUMBER = {3238431},
MRREVIEWER = {J.\ T.\ Davidov},
       DOI = {10.1090/S0002-9939-2014-12128-9},
       URL = {https://doi.org/10.1090/S0002-9939-2014-12128-9},
}

@article {STW17,
    AUTHOR = {Sz\'{e}kelyhidi, G\'{a}bor and Tosatti, Valentino and
              Weinkove, Ben},
     TITLE = {Gauduchon metrics with prescribed volume form},
   JOURNAL = {Acta Math.},
  FJOURNAL = {Acta Mathematica},
    VOLUME = {219},
      YEAR = {2017},
    NUMBER = {1},
     PAGES = {181--211},
      ISSN = {0001-5962,1871-2509},
   MRCLASS = {53C55 (32C36 32Q99)},
  MRNUMBER = {3765661},
MRREVIEWER = {Keizo\ Hasegawa},
       DOI = {10.4310/ACTA.2017.v219.n1.a6},
       URL = {https://doi.org/10.4310/ACTA.2017.v219.n1.a6},
}

@article {TW15,
    AUTHOR = {Tosatti, Valentino and Weinkove, Ben},
     TITLE = {On the evolution of a {H}ermitian metric by its
              {C}hern-{R}icci form},
   JOURNAL = {J. Differential Geom.},
  FJOURNAL = {Journal of Differential Geometry},
    VOLUME = {99},
      YEAR = {2015},
    NUMBER = {1},
     PAGES = {125--163},
      ISSN = {0022-040X,1945-743X},
   MRCLASS = {53C44 (53C55)},
  MRNUMBER = {3299824},
MRREVIEWER = {Chengjie\ Yu},
       URL = {http://projecteuclid.org/euclid.jdg/1418345539},
}

@article {Fuj78,
    AUTHOR = {Fujiki, Akira},
     TITLE = {Closedness of the {D}ouady spaces of compact {K}\"{a}hler
              spaces},
   JOURNAL = {Publ. Res. Inst. Math. Sci.},
  FJOURNAL = {Kyoto University. Research Institute for Mathematical
              Sciences. Publications},
    VOLUME = {14},
      YEAR = {1978},
    NUMBER = {1},
     PAGES = {1--52},
      ISSN = {0034-5318,1663-4926},
   MRCLASS = {32G13},
  MRNUMBER = {486648},
MRREVIEWER = {H.\ Kerner},
       DOI = {10.2977/prims/1195189279},
       URL = {https://doi.org/10.2977/prims/1195189279},
}

@article {Var89,
    AUTHOR = {Varouchas, Jean},
     TITLE = {K\"{a}hler spaces and proper open morphisms},
   JOURNAL = {Math. Ann.},
  FJOURNAL = {Mathematische Annalen},
    VOLUME = {283},
      YEAR = {1989},
    NUMBER = {1},
     PAGES = {13--52},
      ISSN = {0025-5831,1432-1807},
   MRCLASS = {32C15 (32F05 32G10 32H35)},
  MRNUMBER = {973802},
MRREVIEWER = {Siegmund\ Kosarew},
       DOI = {10.1007/BF01457500},
       URL = {https://doi.org/10.1007/BF01457500},
}

@book {OV24,
    AUTHOR = {Ornea, Liviu and Verbitsky, Misha},
     TITLE = {Principles of locally conformally {K}\"{a}hler geometry},
    SERIES = {Progress in Mathematics},
    VOLUME = {354},
 PUBLISHER = {Birkh\"{a}user/Springer, Cham},
      YEAR = {2024},
     PAGES = {xxi+736},
      ISBN = {978-3-031-58119-9; 978-3-031-58120-5},
   MRCLASS = {53C55 (32E10 32Gxx 32Lxx 53-02 53C25)},
  MRNUMBER = {4771164},
       DOI = {10.1007/978-3-031-58120-5},
       URL = {https://doi.org/10.1007/978-3-031-58120-5},
}

@article {GL21,
    AUTHOR = {Guedj, Vincent and Lu, Chinh H.},
     TITLE = {Quasi-plurisubharmonic envelopes 1: uniform estimates on
              {K}\"{a}hler manifolds},
   JOURNAL = {J. Eur. Math. Soc. (JEMS)},
  FJOURNAL = {Journal of the European Mathematical Society (JEMS)},
    VOLUME = {27},
      YEAR = {2025},
    NUMBER = {3},
     PAGES = {1185--1208},
      ISSN = {1435-9855,1435-9863},
   MRCLASS = {32Q15 (32U05 32W20 35A23)},
  MRNUMBER = {4874940},
MRREVIEWER = {Ngoc\ Cuong\ Nguyen},
       DOI = {10.4171/jems/1460},
       URL = {https://doi.org/10.4171/jems/1460},
}

@misc{ALS25,
  author        = {Alehyane, Omar and Lu, Chinh H. and Salouf, Mohammed},
  title         = {Monge--{A}mp\`ere equations with prescribed singularities on compact {H}ermitian manifolds},
  year          = {2025},
  eprint        = {2511.02339},
  archivePrefix = {arXiv},
  primaryClass  = {math.CV},
}

@misc{LiX26,
  author        = {Li, Xuan},
  title         = {Monge--{A}mp\`ere type equations on compact {H}ermitian manifolds with bounded mass property},
  year          = {2026},
  eprint        = {2601.13446},
  archivePrefix = {arXiv},
  primaryClass  = {math.CV},
}

	Mingchen Xia, \textsc{Institute of Geometry and Physics, University of Science and Technology of China}\par\nopagebreak
	\textit{Email address:} \texttt{xiamingchen2008@gmail.com}

	Kewei Zhang, \textsc{School of Mathematical Sciences, Beijing Normal University}\par\nopagebreak
	\textit{Email address:} \texttt{kwzhang@bnu.edu.cn}

\end{document}